\documentclass{daj}

\usepackage{amsmath}
\usepackage{amsthm, amssymb, amsfonts}
\usepackage{bbm}
\usepackage{enumitem}

\newtheorem{theorem}{Theorem}[section]

\newtheorem{lemma}[theorem]{Lemma}

\newtheorem{definition}[theorem]{Definition}
\newtheorem*{definition*}{Definition}

\theoremstyle{remark}
\newtheorem{remark}[theorem]{Remark}

\def\B{\mathcal{B}}

\def \cB{\mathcal B}
\def \cC{\mathcal C}
\def \cD{\mathcal D}
\def \cG{\mathcal G}

\def \bias{{\rm bias}}
\def \arank{{\rm arank}}
\def \prank{{\rm prank}}
\def \s{\subset}
\def\rank{\mathop{\mathrm{rank}}}
\def\e{\epsilon}
\def\E{\mathbb{E}}

\def\C{\mathbb{C}}

\def\P{\mathbb{P}}
\def\F{\mathbb{F}}

\def\a{\alpha}

\def\d{\delta}

\dajAUTHORdetails{%
  title = {Polynomial bound for the Partition Rank vs the Analytic Rank of Tensors}, 
  author = {Oliver Janzer},
  plaintextauthor = {Oliver Janzer},
  keywords = {partition rank, analytic rank, tensor},
}   

\dajEDITORdetails{%
   year={2020},
   number={7},
   received={23 October 2018},   
   revised={6 March 2019},    
   published={19 May 2020},  
   doi={10.19086/da.12935},       
}   

\begin{document}

\begin{frontmatter}[classification=text]

\title{Polynomial bound for the Partition Rank vs the Analytic Rank of Tensors} 

\author[janzer]{Oliver Janzer}

\begin{abstract}
	A tensor defined over a finite field $\F$ has low analytic rank if the distribution of its values differs significantly from the uniform distribution. An order $d$ tensor has partition rank 1 if it can be written as a product of two tensors of order less than $d$, and it has partition rank at most $k$ if it can be written as a sum of $k$ tensors of partition rank 1. In this paper, we prove that if the analytic rank of an order $d$ tensor is at most $r$, then its partition rank is at most $f(r,d,|\F|)$, where, for fixed $d$ and $\F$, $f$ is a polynomial in $r$. This is an improvement of a recent result of the author, where he obtained a tower-type bound. Prior to our work, the best known bound was an Ackermann-type function in $r$ and $d$, though it did not depend on $\F$. It follows from our results that a biased polynomial has low rank; there too we obtain a polynomial dependence improving the previously known Ackermann-type bound.

	A similar polynomial bound for the partition rank was obtained independently and simultaneously by Mili\'cevi\'c.
\end{abstract}
\end{frontmatter}

\section{Introduction}

\subsection{Bias and rank of polynomials}

For a finite field $\F$ and a polynomial $P: \F^n\rightarrow \F$, we say that $P$ is unbiased if the distribution of the values $P(x)$ is close to the uniform distribution on $\F$; otherwise we say that $P$ is biased. It is an important direction of research in higher order Fourier analysis to understand the structure of biased polynomials.

Note that a generic degree $d$ polynomial should be unbiased. In fact, as we will see below, if a degree $d$ polynomial is biased, then it can be written as a function of not too many polynomials of degree at most $d-1$. Let us now make this discussion more precise.

\begin{definition}
	Let $\F$ be a finite field and let $\chi$ be a nontrivial character of $\F$. The \emph{bias} of a function $f:\F^n\rightarrow \F$ with respect to $\chi$ is defined to be $\bias_{\chi}(f)=\E_{x\in \F^{n}} \lbrack \chi(f(x)) \rbrack$. (Here and elsewhere in the paper $\E_{x\in G}h(x)$ denotes $\frac{1}{|G|}\sum_{x\in G}h(x)$.)
\end{definition}

\begin{remark}
	Most of the previous work is on the case $\F=\F_p$ with $p$ a prime, in which case the standard definition of bias is $\bias(f)=\E_{x\in \F^n} \omega^{f(x)}$ where $\omega=e^{\frac{2\pi i}{p}}$.
\end{remark}

\begin{definition}
	Let $P$ be a polynomial $\F^n\rightarrow \F$ of degree $d$. The \emph{rank} of $P$ (denoted $\rank(P)$) is defined to be the smallest integer $r$ such that there exist polynomials $Q_1,\dots,Q_r:\F^n\rightarrow \F$ of degree at most $d-1$ and a function $f:\F^r\rightarrow \F$ such that $P=f(Q_1,\dots,Q_r)$.
\end{definition}

As discussed above, it is known that if a polynomial has large bias, then it has low rank. The first result in this direction was proved by Green and Tao \cite{greentao} who showed that if $\F$ is a field of prime order and $P:\F^n\rightarrow \F$ is a polynomial of degree $d$ with $d<|\F|$ and $\bias(P)\geq \d>0$, then $\rank(P)\leq c(\F,\d,d)$. Kaufman and Lovett \cite{kaufmanlovett} proved that the condition $d<|\F|$ can be omitted. In both results, $c$ has Ackermann-type dependence on its parameters. Finally, Bhowmick and Lovett \cite{bhowmicklovett} proved that if $d<\text{char}(\F)$ and $\bias(P)\geq |\F|^{-s}$, then $\rank(P)\leq c'(d,s)$. The novelty of this result is that $c'$ does not depend on $\F$. However, it still has Ackermann-type dependence on $d$ and $s$.

One of our main results is the following theorem, which improves the result of Bhowmick and Lovett, unless $|\F|$ is very large.

\begin{theorem} \label{biaslowrank}
	Let $\F$ be a finite field and let $\chi$ be a nontrivial character of $\F$. Let $P$ be a polynomial $\F^n\rightarrow \F$ of degree $d<\text{char}(\F)$. Suppose that $\bias_{\chi}(P)\geq \e>0$ where $\e\leq 1/|\F|$. Then $$\rank(P)\leq (c\cdot 2^d \cdot \log(1/\e))^{c'(d)}+1$$
	where $c$ is an absolute constant and $c'(d)=4^{d^d}$.
\end{theorem}

Recall that if $G$ is an Abelian group and $d$ is a positive integer, then the Gowers $U^d$ norm (which is only a seminorm for $d=1$) of $f:G\rightarrow \C$ is defined to be
$$\|f\|_{U^{d}}=\big|\E_{x,y_1,\dots,y_d\in G} \prod_{S\s \lbrack d\rbrack} \cC^{d-|S|} f(x+\sum_{i\in S}y_i)\big|^{1/2^d},$$
where $\cC$ is the conjugation operator. It is a major area of research to understand the structure of functions $f$ whose $U^d$ norm is large. Our next theorem is a result in this direction.

\begin{theorem} \label{largenormlowrank}
	Let $\F$ be a finite field and let $\chi$ be a nontrivial character of $\F$. Let $P$ be a polynomial $\F^n\rightarrow \F$ of degree $d<\text{char}(\F)$. Let $f(x)=\chi(P(x))$ and assume that $\|f\|_{U^{d}}\geq \e>0$ where $\e\leq 1/|\F|$. Then $$\rank(P)\leq (c\cdot 2^d \cdot \log(1/\e))^{c'(d)}+1$$
	where $c$ is an absolute constant and $c'(d)=4^{d^d}$.
\end{theorem}

Our result implies a similar improvement to the bounds for the quantitative inverse theorem for Gowers norms for polynomial phase functions of degree $d$.

\begin{theorem} \label{inverse}
	Let $\F$ be a field of prime order and let $P$ be a polynomial $\F^n\rightarrow \F$ of degree $d<\text{char}(\F)$. Let $f(x)=\omega^{P(x)}$ where $\omega=e^{\frac{2\pi i}{|\F|}}$ and assume that $\|f\|_{U^d}\geq \e>0$ where $\e\leq 1/|\F|$. Then there exists a polynomial $Q:\F^n\rightarrow \F$ of degree at most $d-1$ such that $$|\E_{x\in \F^n} \omega^{P(x)}\overline{\omega^{Q(x)}}|\geq |\F|^{-(c\cdot 2^d \cdot \log(1/\e))^{c'(d)}-1}$$
	where $c$ is an absolute constant and $c'(d)=4^{d^d}$.
\end{theorem}

Theorems \ref{biaslowrank} and \ref{inverse} easily follow from Theorem \ref{largenormlowrank}.

\smallskip

\emph{Proof of Theorem \ref{biaslowrank}.} Note that when $f(x)=\chi(P(x))$, then $\|f\|^2_{U^1}=|\E_{x,y\in \F^n}\overline{f(x)}f(x+y)|=|\E_{x\in \F^n}f(x)|^2$, so $\|f\|_{U^1}=|E_{x\in \F^n}f(x)|=|\bias_{\chi}(P)|$. However, $\|f\|_{U^k}$ is increasing in $k$ (see eg. Claim 6.2.2 in \cite{fouriersurvey}), therefore $\|f\|_{U^d}\geq |\bias_{\chi}(P)|\geq \e$. The result is now immediate from Theorem \ref{largenormlowrank}.

\smallskip

\emph{Proof of Theorem \ref{inverse}.} By Theorem \ref{largenormlowrank}, there exists a set of $r\leq (c\cdot 2^d \cdot \log(1/\e))^{c'(d)}+1$  polynomials $Q_1,\dots,Q_r$ such that $P(x)$ is a function of $Q_1(x),\dots,Q_r(x)$. \newline Then $\omega^{P(x)}=g(Q_1(x),\dots,Q_r(x))$ for some function $g:\F^r\rightarrow \C$. Let $G=\F^r$. Note that $|g(y)|=1$ for all $y\in G$, therefore $|\hat{g}(\chi)|\leq 1$ for every character $\chi\in \hat{G}$. Now $\omega^{P(x)}=\sum_{\chi\in \hat{G}} \hat{g}(\chi)\,\chi((Q_1(x),\dots,Q_r(x))$, so $$1=\E_{x\in \F^n} |\omega^{P(x)}|^2=\sum_{\chi \in \hat{G}} \overline{\hat{g}(\chi)}\bigg( \E_{x\in \F^n} \omega^{P(x)}\overline{\chi(Q_1(x),\dots,Q_r(x))} \bigg).$$ Thus, there exists some $\chi\in \hat{G}$ with $|\E_{x\in \F^n} \omega^{P(x)}\overline{\chi(Q_1(x),\dots,Q_r(x))}|\geq 1/|G|=1/|\F|^r$. But $\chi$ is of the form $\chi(y_1,\dots,y_r)=\omega^{\sum_{i\leq r} \a_i y_i}$ for some $\a_i\in \F$. Then $\chi(Q_1(x),\dots,Q_r(x))=\omega^{Q_{\a}(x)}$, where $Q_{\a}$ is the degree $d-1$ polynomial $Q_{\a}(x)=\sum_{i\leq r} \a_iQ_i(x)$. So $Q=Q_{\a}$ is a suitable choice.

\subsection{Analytic rank and partition rank of tensors}

Related to the bias and rank of polynomials are the notions of \emph{analytic rank} and \emph{partition rank} of tensors. Recall that if $\F$ is a field and $V_1,\dots,V_d$ are finite dimensional vector spaces over $\F$, then an order $d$ tensor is a multilinear map $T:V_1\times \dots \times V_d\rightarrow \F$. (In this subsection, assume that $d\geq 2$.) Each $V_k$ can be identified with $\F^{n_k}$ for some $n_k$, and then there exist $t_{i_1,\dots,i_d}\in \F$ for all $i_1\leq n_1,\dots,i_d\leq n_d$ such that $T(v^1,\dots,v^d)=\sum_{i_1\leq n_1,\dots,i_d\leq n_d} t_{i_1,\dots,i_d} v^1_{i_1}\dots v^d_{i_d}$ for every $v^1\in \F^{n_1},\dots,v^d\in \F^{n_d}$ (where $v_k$ is the $k$th coordinate of the vector $v$). Indeed, $t_{i_1,\dots,i_d}$ is just $T(e^{i_1},\dots,e^{i_d})$, where $e^i$ is the $i$th standard basis vector.

The following notion was introduced by Gowers and Wolf \cite{gowerswolf}.

\begin{definition}
	Let $\F$ be a finite field, let $V_1,\dots,V_d$ be finite dimensional vector spaces over $\F$ and let $T:V_1\times \dots \times V_d\rightarrow \F$ be an order $d$ tensor. Then the \emph{analytic rank} of $T$ is defined to be $\arank(T)=-\log_{|\F|} \bias(T)$, where $\bias(T)=\E_{v^1\in V_1,\dots,v^d\in V_d} \lbrack \chi(T(v^1,\dots,v^d))\rbrack$ for any nontrivial character $\chi$ of $\F$.
\end{definition}

\begin{remark} \label{welldefined}
	This is well-defined. Indeed, if $\chi$ is a nontrivial character of $\F$, then
	\begin{align*}
	\E_{v^1\in V_1,\dots,v^d\in V_d} \lbrack \,\chi(T(v^1,\dots,v^d))\rbrack&=\E_{v^1\in V_1,\dots,v^{d-1}\in V_{d-1}}\lbrack \E_{v^d\in V_d} \,\chi(T(v^1,\dots,v^d))\rbrack \\
	&=\P_{v^1\in V_1,\dots,v^{d-1}\in V_{d-1}} \lbrack T(v^1,\dots,v^{d-1},x)\equiv 0\rbrack,
	\end{align*}
	where $T(v^1,\dots,v^{d-1},x)$ is viewed as a function in $x$. The second equality holds because \newline $\E_{v^d\in V_d} \,\chi(T(v^1,\dots,v^d))=0$ unless $T(v^1,\dots,v^{d-1},x)\equiv 0$, in which case it is 1.
	
	Thus, $\E_{v^1\in V_1,\dots,v^d\in V_d} \lbrack \,\chi(T(v^1,\dots,v^d))\rbrack$ does not depend on $\chi$, and is always positive. Moreover, it is at most 1, therefore the analytic rank is always nonnegative.
\end{remark}

\smallskip

A different notion of rank was defined by Naslund \cite{naslund}.

\begin{definition}
	Let $T:V_1\times \dots \times V_d\rightarrow \F$ be a (non-zero) order $d$ tensor. We say that $T$ has \emph{partition rank} 1 if there is some $S\s \lbrack d\rbrack$ with $S\neq \emptyset,S\neq \lbrack d\rbrack$ such that $T(v^1,\dots,v^d)=T_1(v^{i}:i\in S)T_2(v^{i}:i\not \in S)$ where $T_1:\prod_{i\in S}V_i\rightarrow \F,T_2:\prod_{i\not\in S}V_i\rightarrow \F$ are tensors. In general, the partition rank of $T$ is the smallest $r$ such that $T$ can be written as the sum of $r$ tensors of partition rank 1. This number is denoted $\prank(T)$.
\end{definition}

Kazhdan and Ziegler \cite{KZ18} and Lovett \cite{lovett} proved that $\arank(T)\leq \prank(T)$. In the other direction, it follows from the work of Bhowmick and Lovett \cite{bhowmicklovett} that if an order $d$ tensor $T$ has $\arank(T)\leq r$, then $\prank(T)\leq f(r,d)$ for some function $f$. Note that $f$ does not depend on $|\F|$ or the dimension of the vector spaces $V_k$. However, $f$ has an Ackermann-type dependence on $d$ and $r$. For $d=3,4$, better bounds were established by Haramaty and Shpilka \cite{HS10}. They proved that for $d=3$ we have $\prank(T)=O(r^4)$, and that for $d=4$ we have $\prank(T)=\exp(O(r))$.

The main result of our paper is a polynomial upper bound, which holds for general $d$.


\begin{theorem} \label{analyticpartition}
	Let $T:V_1\times \dots \times V_d\rightarrow \F$ be an order $d$ tensor with $\arank(T)\leq r$ and assume that $r\geq 1$. Then
	$$\prank(T)\leq (c\cdot \log |\F|)^{{c'(d)}} \cdot r^{c'(d)}$$ for some absolute constant $c$, and $c'(d)=4^{d^d}$.
\end{theorem}

We remark that a very similar result was obtained independently and simultaneously by Mili\'cevi\'c \cite{Mil19}. Moreover, in the special case $d=4$, a similar bound was proved independently by Lampert \cite{Lam19}.

\medskip

It is not hard to see that Theorem \ref{analyticpartition} implies Theorem \ref{largenormlowrank}. Indeed, let $P$ be a polynomial $\F^n\rightarrow \F$ of degree $d<\text{char}(\F)$, let $f(x)=\chi(P(x))$ and assume that $\|f\|_{U^d}\geq \e>0$, where $\e\leq 1/|\F|$. Define $T:(\F^n)^d \rightarrow \F$ by $T(y_1,\dots,y_d)=\sum_{S\s \lbrack d\rbrack} (-1)^{d-|S|}P(\sum_{i\in S} y_i)$. By Lemma 2.4 from \cite{gowerswolf}, $T$ is a tensor of order $d$. Moreover, by the same lemma, we have $T(y_1,\dots,y_d)=\sum_{S\s \lbrack d\rbrack} (-1)^{d-|S|}P(x+\sum_{i\in S} y_i)$ for any $x\in \F^n$. Thus, $$\bias(T)=\E_{y_1,\dots,y_d\in \F^n} \, \chi(T(y_1,\dots,y_d))=\E_{y_1,\dots,y_d\in \F^n} \prod_{S\s \lbrack d\rbrack} \cC^{d-|S|}f(x+\sum_{i\in S} y_i)$$ for any $x\in \F^n$. By averaging over all $x\in \F^n$, it follows that $\bias(T)=\|f\|^{2^d}_{U^d}\geq \e^{2^d}$. Thus, $\arank(T)\leq 2^d\log_{|\F|}(1/\e)$. Note that $2^d\log_{|\F|}(1/\e)\geq 1$. Therefore, by Theorem \ref{analyticpartition} with $r=2^d\log_{|\F|}(1/\e)$, we get
\begin{equation} \label{eqnprank}
\prank(T)\leq (c\cdot 2^d \cdot \log(1/\e))^{c'(d)}.
\end{equation}
Note that $T(y_1,\dots,y_d)=D_{y_1}\dots D_{y_d}P(x)$ where $D_yg(x)=g(x+y)-g(x)$. Thus, by Taylor's approximation theorem, since $d<\text{char}(\F)$, we get $P(x)=\frac{1}{d!}D_x \dots D_x P(0)+W(x)=\frac{1}{d!}T(x,\dots,x)+W(x)$ for some polynomial $W$ of degree at most $d-1$.

By equation (\ref{eqnprank}), $T$ can be written as a sum of at most $(c\cdot 2^d \cdot \log(1/\e))^{c'(d)}$ tensors of partition rank 1. Hence, $\frac{1}{d!}T(x,\dots,x)$ can be written as a sum of at most $(c\cdot 2^d \cdot \log(1/\e))^{c'(d)}$ expressions of the form $Q(x)R(x)$ where $Q,R$ are polynomials of degree at most $d-1$ each. Thus, $P-W$ has rank at most $(c\cdot 2^d \cdot \log(1/\e))^{c'(d)}$, and therefore $P$ has rank at most $$(c\cdot 2^d \cdot \log(1/\e))^{c'(d)}+1.$$

We remark that the proof of the main result of this paper follows the strategy introduced by the author in \cite{janzerrank}, but the argument is improved locally at a few places.

\section{The proof of Theorem \ref{analyticpartition}}

\subsection{Notation and preliminaries}

In the rest of the paper, we identify $V_i$ with $\F^{n_i}$. Thus, the set of all tensors $V_1\times \dots \times V_d\rightarrow \F$ is the tensor product $\F^{n_1}\otimes \dots \otimes \F^{n_d}$, which will be denoted by $\cG$ throughout this section. Also, $\cB$ will always stand for the multiset $\{u_1\otimes\dots\otimes u_d:u_i\in \F^{n_i} \text{ for all } i\}$. The elements of $\cB$ will be called \emph{pure tensors}. Note that $\cG=\F^{n_1}\otimes\dots\otimes\F^{n_d}$ can be viewed as the set of $d$-dimensional $(n_1,\dots,n_d)$-arrays over $\F$ which in turn can be viewed as $\F^{n_1n_2\dots n_d}$, equipped with the entry-wise dot product.

For $I\s \lbrack d\rbrack$, we write $\F^{I}$ for $\bigotimes_{i\in I} \F^{n_i}$ so that we naturally have $\cG=\F^{I}\otimes \F^{I^c}$, where $I^c$ always denotes $\lbrack d\rbrack \setminus I$.

If $r\in \F^{\lbrack d\rbrack}=\cG$ and $s\in \F^{\lbrack k\rbrack}$ (for some $k\leq d$), then we define $rs$ to be the tensor in $\F^{\lbrack k+1,d\rbrack}$ with coordinates $(rs)_{i_{k+1},\dots,i_d}=\sum_{i_1\leq n_1,\dots,i_k\leq n_k} r_{i_1,\dots,i_d}s_{i_1,\dots,i_k}$. If $k=d$, then $rs$ is the same as the entry-wise dot product $r.s$. Also, note that viewing $r$ as a $d$-multilinear map $R:\F^{n_1}\times \dots \times \F^{n_d}\rightarrow \F$, we have $R(v^1,\dots,v^d)=\sum_{i_1\leq n_i,\dots,i_d\leq n_d} r_{i_1,\dots,i_d}v^1_{i_1}\dots v^d_{i_d}=r(v^1\otimes \dots \otimes v^d)$.

Finally, we use a non-standard notation and write $kB$ to mean the set of elements of $\cG$ which can be written as a sum of \textit{at most} $k$ elements of $B$, where $B$ is some fixed (multi)subset of $\cG$, and similarly, we write $kB-lB$ for the set of elements that can be obtained by adding at most $k$ members and subtracting at most $l$ members of $B$.

We will use the next result several times in our proofs. It is a version of Bogolyubov's lemma, due to Sanders.

\begin{lemma}[Sanders \cite{sanders}] \label{lemmabogolyubov}
	There is an absolute constant $C$ with the following property. Let $A$ be a subset of $\F^n$ with $|A|\geq \d |\F^n|$. Then $2A-2A$ contains a subspace of $\F^{n}$ of codimension at most $C(\log (1/\d))^4$.
\end{lemma}

Throughout the paper, $C$ stands for the constant appearing in the previous lemma. Clearly we may assume that $C\geq 1$. Logarithms are base 2.

\subsection{The main lemma and some consequences}

Theorem \ref{analyticpartition} will follow easily from the next lemma, which is the main technical result of this paper. See \cite{gowersjanzer} for an application of a qualitative version of this lemma.

\begin{lemma} \label{mainlemma}
	Let $d\geq 1$ be an integer and let $\d\leq 1/2$. Let $f_1(d)=2^{3^{d+3}}$, $f_2(d)=2^{-3^{d+3}}$ and $G(d,\d,\F)=((\log |\F|)c_1(d)(\log 1/\d))^{c_2(d)}$ where $c_1(d)=C\cdot 2^{3^{d+6}}$ and $c_2(d)=4^{d^{d}}$. If $\cB'\subset \cB$ is a multiset such that $|\cB'|\geq \d |\cB|$, then there exists a multiset $Q$ whose elements are pure tensors chosen from $f_1(d)\cB'-f_1(d)\cB'$ (but with arbitrary multiplicity) with the following property. The set of arrays $r\in \cG$ with $r.q=0$ for at least $(1-f_2(d))|Q|$ choices $q\in Q$ is contained in $\sum_{I\subset \lbrack d\rbrack, I\neq \emptyset} V_I\otimes \F^{I^c}$ for subspaces $V_I\subset \F^{I}$ of dimension at most $G(d,\d,\F)$.
\end{lemma}

Throughout the paper, the functions $G,c_1,c_2$ will refer to the functions introduced in the previous lemma. In fact, as $\F$ is fixed, we will write $G(d,\d)$ to mean $G(d,\d,\F)$.

In this subsection we deduce Theorem \ref{analyticpartition} from Lemma \ref{mainlemma}.


The notion introduced in the next definition is closely related to the partition rank, but will be somewhat more convenient to work with.

\begin{definition}
	Let $k$ be a positive integer. We say that $r\in \cG$ is \emph{$k$-degenerate} if for every $I\s \lbrack d\rbrack,I\neq \emptyset,I\neq \lbrack d\rbrack$, there exists a subspace $H_I\s \F^{I}$ of dimension at most $k$ such that $r\in \sum_{I\subset \lbrack d-1 \rbrack,I\neq \emptyset}H_I\otimes H_{I^c}$.
\end{definition}

If $r\in H_I\otimes \F^{I^c}$ with $\dim(H_I)\leq k$, then $r\in H_I\otimes H_{I^c}$ for some $H_{I^c}\s \F^{I^c}$ of dimension at most $k$. (This follows by writing $r$ as $\sum_{j\leq m} s_j\otimes t_j$ with $\{s_j\}$ a basis for $H_I$ and letting $H_{I^c}$ be the span of all the $t_j$.) Thus, $r$ is $k$-degenerate if and only if $r\in \sum_{I\s \lbrack d-1 \rbrack, I\neq \emptyset} H_I\otimes \F^{I^c}$ for some $H_I\s \F^{I}$ of dimension at most $k$, or equivalently, if and only if $r\in \sum_{I\s \lbrack d-1 \rbrack, I\neq \emptyset} \F^{I}\otimes H_{I^c}$ for some $H_{I^c}\s \F^{I^c}$ of dimension at most $k$. Moreover, note that if $r$ is $k$-degenerate, then $\prank(r)\leq 2^{d-1}k$. This is because if $I\neq \emptyset,I\s \lbrack d-1 \rbrack$ and $w\in H_I\otimes H_{I^c}$ for subspaces $H_I\s \F^{I}$ and $H_{I^c}\s \F^{I^c}$ of dimension at most $k$, then $w=\sum_{i\leq k} s_i\otimes t_i$ for some $s_i\in H_{I}$, $t_i\in H_{I^c}$. But clearly, $s_i\otimes t_i$ has partition rank 1.

\begin{lemma} \label{biastorank}
	Let $\d\leq 1/2$ and $d\geq 2$. Suppose that Lemma \ref{mainlemma} has been proved for $d'=d-1$. Let $r\in \cG$ be such that $r(v_1\otimes \dots \otimes v_{d-1})=0\in \F^{n_d}$ for at least $\d|\F|^{n_1\dots n_{d-1}}$ choices $v_1\in \F^{n_1},\dots,v_{d-1}\in \F^{n_{d-1}}$. Then $r$ is $f$-degenerate for $f=G(d-1,\d)$.
\end{lemma}

\begin{proof}
	Write $r=\sum_i s_i\otimes t_i$ where $s_i\in \F^{\lbrack d-1\rbrack}$ and $\{t_i\}_i$ is a basis for $\F^{n_d}$. Let $\cD$ be the multiset $\{u_1\otimes \dots \otimes u_{d-1}:u_1\in \F^{n_1},\dots,u_{d-1}\in\F^{n_{d-1}}\}$ and let $\cD'=\{w\in \cD: rw=0\}$. Since $|\cD'|\geq \d|\cD|$, by Lemma \ref{mainlemma} there is a multiset $Q$ with elements from $2^{3^{d+2}}\cD'-2^{3^{d+2}}\cD'$ such that the set of arrays $r'\in \F^{\lbrack d-1\rbrack}$ with $r'.q=0$ for all choices $q\in Q$ is contained in some $\sum_{I\s \lbrack d-1\rbrack,I\neq \emptyset} V_I\otimes \F^{\lbrack d-1\rbrack \setminus I}$, where $\dim(V_I)\leq G(d-1,\d)$. Note that for every $i$ we have $s_i.w=0$ for all $w\in \cD'$ and so also $s_i.q=0$ for all $q\in Q$. Thus, $r\in \sum_{I\s \lbrack d-1\rbrack,I\neq \emptyset} V_I\otimes \F^{I^c}$.
\end{proof}

Now we are in a position to prove Theorem \ref{analyticpartition} conditional on Lemma \ref{mainlemma}.

\begin{proof}[Proof of Theorem \ref{analyticpartition}]
	Let $T:\F^{n_1}\times \dots \times \F^{n_d}\rightarrow \F$ be an order $d$ tensor with $\arank(T)\leq r$. By Remark \ref{welldefined}, we have $\P_{v_1\in \F^{n_1},\dots,v_{d-1}\in \F^{n_{d-1}}}\lbrack T(v_1,\dots,v_{d-1},x)\equiv 0\rbrack\geq |\F|^{-r}$. Writing $t$ for the element in $\cG$ corresponding to $T$, we get that $t(v_1\otimes \dots \otimes v_{d-1}\otimes x)\equiv 0$ as a function of $x$ for at least $\d|\F|^{n_1\dots n_d}$ choices $v_1\in \F^{n_1},\dots,v_{d-1}\in \F^{n_{d-1}}$, where $\d=|\F|^{-r}$. But $t(v_1\otimes \dots \otimes v_{d-1}\otimes x)=\big(t(v_1\otimes \dots \otimes v_{d-1})\big).x$, so we have $t(v_1\otimes \dots \otimes v_{d-1})=0$ for all these choices of $v_i$. The condition $r\geq 1$ implies $\d\leq 1/2$, therefore by Lemma \ref{biastorank}, $t$ is $f$-degenerate for $f=G(d-1,\d)$. Hence,
	
	\begin{align*} \prank(T)&\leq 2^{d-1}G(d-1,\d) \\ &=2^{d-1}((\log |\F|)\cdot c_1(d-1) \cdot \log (|\F|^r))^{c_2(d-1)} \\
	&=2^{d-1}((\log |\F|)^2\cdot c_1(d-1)\cdot r)^{c_2(d-1)} \\
	&\leq ((\log |\F|)^2\cdot c_1(d)\cdot r)^{c_2(d-1)}
	\end{align*}
	But there exists some absolute constant $c$ such that $c_1(d)^{c_2(d-1)}\leq c^{c_2(d)}$ holds for all $d$. Moreover, $2c_2(d-1)\leq c_2(d)$. Thus, $\prank(T)\leq (c\cdot \log |\F|)^{c_2(d)}\cdot r^{c_2(d)}=(c\cdot \log |\F|)^{c'(d)}\cdot r^{c'(d)}$.
\end{proof}

\subsection{The overview of the proof of Lemma \ref{mainlemma}} \label{subsectionoverviewranklemma}

The proof of the lemma goes by induction on $d$. In what follows, we shall prove results conditional on the assumption that Lemma \ref{mainlemma} has been verified for all $d'<d$. Eventually, we will use these results to prove the induction step.

In this subsection, we give a detailed sketch of the proof in the $d=3$ case. At the end of the subsection, we also briefly sketch the $d>3$ case.

\subsubsection{The high-level outline in the case $d=3$}

We assume that Lemma \ref{mainlemma} has been proven for $d\leq 2$ and use this assumption to show that it holds for $d=3$. We will take $Q=Q_{\{1,2,3\}}\cup Q_{\{1\}}\cup Q_{\{2\}}\cup Q_{\{3\}}$ with elements chosen from $2^{3^{d+3}}\cB'-2^{3^{d+3}}\cB'$ such that the $Q_I$ have roughly equal size. This implies that if for some $r\in \cG$ we have $r.q=0$ for almost all $q\in Q$, then $r.q=0$ holds for almost all $q\in Q_I$ for every $I=\{1\},\{2\},\{3\},\{1,2,3\}$. We define $Q_{\{1,2,3\}}$ first, in a way that if $r.q=0$ for almost all $q\in Q_{\{1,2,3\}}$, then $r=x+y$ where $x\in V_{\{1,2,3\}}$ for a vector space $V_{\{1,2,3\}}$ which is independent of $r$ and have small dimension, and $y$ has small partition rank. This already implies that any array $r\in \cG$ with $r.q=0$ for almost all $q\in Q$ is contained in $V_{\{1,2,3\}}+\F^{n_1}\otimes H_{\{2,3\}}(r)+\F^{n_2}\otimes H_{\{1,3\}}(r)+\F^{n_3}\otimes H_{\{1,2\}}(r)$ for some subspaces $H_I(r)\s \F^{I}$ depending on $r$ and of small dimension. We then find $Q_{\{1\}}$ such that if $r\in V_{\{1,2,3\}}+\F^{n_1}\otimes H_{\{2,3\}}(r)+\F^{n_2}\otimes H_{\{1,3\}}(r)+\F^{n_3}\otimes H_{\{1,2\}}(r)$ has $r.q=0$ for almost all $q\in Q_{\{1\}}$, then $r\in V_{\{1,2,3\}}+V_{\{1\}}\otimes \F^{\{2,3\}}+\F^{n_1}\otimes V_{\{2,3\}}+\F^{n_2}\otimes K_{\{1,3\}}(r)+\F^{n_3}\otimes K_{\{1,2\}}(r)$, where $V_{\{1\}}\s \F^{n_1}$ and $V_{\{2,3\}}\s \F^{\{2,3\}}$ are subspaces independent of $r$ and have small dimension, and $K_I(r)\s \F^{I}$ are subspaces of small dimension (although quite a bit larger than $\dim(H_I(r))$). Then we find $Q_{\{2\}}$ such that if $r\in V_{\{1,2,3\}}+V_{\{1\}}\otimes \F^{\{2,3\}}+\F^{n_1}\otimes V_{\{2,3\}}+\F^{n_2}\otimes K_{\{1,3\}}(r)+\F^{n_3}\otimes K_{\{1,2\}}(r)$ has $r.q=0$ for almost all $q\in Q_{\{2\}}$, then $r\in V_{\{1,2,3\}}+V_{\{1\}}\otimes \F^{\{2,3\}}+\F^{n_1}\otimes V_{\{2,3\}}+V_{\{2\}}\otimes \F^{\{1,3\}}+\F^{n_2}\otimes V_{\{1,3\}}+\F^{n_3}\otimes L_{\{1,2\}}(r)$, where $V_{\{2\}}\s \F^{n_2}$ and $V_{\{1,3\}}\s \F^{\{1,3\}}$ are subspaces independent of $r$ and have small dimension, and $L_{\{1,2\}}(r)\s \F^{\{1,2\}}$ is a subspace of small dimension. Finally, we find $Q_{\{3\}}$ such that if $r\in V_{\{1,2,3\}}+V_{\{1\}}\otimes \F^{\{2,3\}}+\F^{n_1}\otimes V_{\{2,3\}}+V_{\{2\}}\otimes \F^{\{1,3\}}+\F^{n_2}\otimes V_{\{1,3\}}+\F^{n_3}\otimes L_{\{1,2\}}(r)$ has $r.q=0$ for almost all $q\in Q_{\{3\}}$, then $r\in V_{\{1,2,3\}}+V_{\{1\}}\otimes \F^{\{2,3\}}+\F^{n_1}\otimes V_{\{2,3\}}+V_{\{2\}}\otimes \F^{\{1,3\}}+\F^{n_2}\otimes V_{\{1,3\}}+V_{\{3\}}\otimes \F^{\{1,2\}}+ \F^{n_3}\otimes V_{\{1,2\}}$, where $V_{\{3\}}\s \F^{n_3}$ and $V_{\{1,2\}}\s \F^{\{1,2\}}$ are subspaces independent of $r$ and have small dimension.

How will we find $Q_{\{1,2,3\}},Q_{\{1\}},Q_{\{2\}}$ and $Q_{\{3\}}$? In this outline we will only explain how to find $Q_{\{2\}}$ (but finding $Q_{\{1\}}$ and $Q_{\{3\}}$ is very similar). We take $Q_{\{2\}}=\bigcup_{u\in U} u\otimes Q_u$ where $U\s \F^{n_2}$ is a subspace of low codimension, and for each $u\in U$, $Q_u\s \F^{\{1,3\}}$ is a multiset consisting of pure tensors such that if for some $x\in \F^{\{1,3\}}$ we have $x.t=0$ for almost all $t\in Q_u$, then $x\in W_{\{1,3\}}(u)+\F^{n_1}\otimes W_{\{3\}}(u)+W_{\{1\}}(u)\otimes \F^{n_3}$ for some subspaces $W_I(u)\s \F^{I}$ not depending on $x$ and of small dimension. Let us call a $Q_u$ with this property \emph{forcing}. We will also make sure that all the $Q_u$ have roughly the same size.

\subsubsection{Why does this $Q_{\{2\}}$ work?} \label{subsubwhyworks}

In what follows, we will sketch why this choice is suitable. 
We remark that in the general case this is done in Lemma \ref{keylemma}. Let $R$ consist of those $$r\in V_{\{1,2,3\}}+V_{\{1\}}\otimes \F^{\{2,3\}}+\F^{n_1}\otimes V_{\{2,3\}}+\F^{n_2}\otimes K_{\{1,3\}}(r)+\F^{n_3}\otimes K_{\{1,2\}}(r)$$ such that $r.q=0$ for almost all $q\in Q_{\{2\}}$. Let $r\in R$. Write $r=r_2+r_3+r_4$ where $$r_2\in V_{\{1\}}\otimes \F^{\{2,3\}}+\F^{n_1}\otimes V_{\{2,3\}}+\F^{n_3}\otimes K_{\{1,2\}}(r), \hspace{3mm} r_3\in V_{\{1,2,3\}}, \hspace{3mm} r_4\in \F^{n_2}\otimes K_{\{1,3\}}(r).$$ It is enough to prove that
\begin{equation}
r_4\in V_{\{2\}}\otimes \F^{\{1,3\}}+\F^{n_2}\otimes V_{\{1,3\}}+\F^{n_3}\otimes L'_{\{1,2\}}(r) \label{eqnr4}
\end{equation} for some small subspaces $V_{\{2\}}\s \F^{n_2}$, $V_{\{1,3\}}\s \F^{\{1,3\}}$ and $L'_{\{1,2\}}(r)\s \F^{\{1,2\}}$ (in fact, we will be able to take $V_{\{2\}}=U^{\perp}$).

First note that $r_2u$ has small (partition) rank for every $u\in U$. Indeed, $r_2u\in V_{\{1\}}\otimes \F^{n_3}+\F^{n_1}\otimes V_{\{2,3\}}u+\F^{n_3}\otimes K_{\{1,2\}}(r)u$, where, for a vector space $L$ of tensors, $Lu$ denotes the space $\{su: s\in L\}$.

Moreover, since the $Q_u$ all have roughly the same size, for almost every $u\in U$ we have that $r.(u\otimes t)=0$ holds for almost every $t\in Q_u$. But $r.(u\otimes t)=(ru).t$, therefore as $Q_u$ is forcing, it follows that for any such $u$ $$ru\in W_{\{1,3\}}(u)+\F^{n_1}\otimes W_{\{3\}}(u)+W_{\{1\}}(u)\otimes \F^{n_3}$$ for some subspaces $W_I(u)\s \F^{I}$ not depending on $r$ and of small dimension.  Since any element of $\F^{n_1}\otimes W_{\{3\}}(u)+W_{\{1\}}(u)\otimes \F^{n_3}$ has small partition rank, it follows that for almost every $u\in U$,
\begin{equation}
r_4u=ru-r_2u-r_3u\in W_{\{1,3\}}(u)+V_{\{1,2,3\}}u+s(u) \label{eqnclosetospace}
\end{equation} where $s(u)$ is a tensor of small partition rank.

Define a sequence $0=Z(0)\s Z(1)\s \dots \s Z(m)\s \F^{\{1,3\}}$ of subspaces recursively as follows. Given $Z(j)$, if there is some $r\in R$ such that $r_4u$ is far from $Z(j)$ for many $u\in U$, then set $Z(j+1)=Z(j)+K_{1,3}(r)$. What we mean by $r_4u$ being far from $Z(j)$ is that there is no $z\in Z(j)$ such that $r_4u-z$ has small partition rank. For suitably chosen parameters, one can show that this procedure cannot go on for too long, ie. that for some not too large $m$ we have that for every $r\in R$, for almost all $u\in U$ there is some $z\in Z(m)$ with $r_4u-z$ having small partition rank.

Now let $r\in R$. Let $X(r)$ be the set consisting of those $x\in K_{\{1,3\}}(r)$ which are close to $Z(m)$. Then $r_4u\in X(r)$ for almost every $u\in U$. Let $t_1,\dots,t_{\a}$ be a maximal linearly independent subset of $X(r)$ and extend it to a basis $t_1,\dots,t_{\a},t'_1,\dots,t'_{\beta}$ for $K_{\{1,3\}}(r)$. Now if a linear combination of $t_1,\dots,t_{\a},t'_1,\dots,t'_{\beta}$ is in $X(r)$, then the coefficients of $t'_1,\dots,t'_{\beta}$ are all zero. Write $r_4=\sum_{i\leq \a}s_i\otimes t_i+\sum_{j\leq \beta}s'_j\otimes t'_j$ for some $s_i,s'_j\in \F^{n_2}$. Since $r_4u\in X(r)$ for almost all $u\in U$, we have, for all $j$, that $s'_j.u=0$ for almost all $u\in U$. Since these hold for more than half of $u\in U$, we obtain $s'_j\in U^{\perp}$ for every $j$, therefore $\sum_{j\leq \beta} s'_j\otimes t'_j\in U^{\perp}\otimes \F^{\{1,3\}}$.

Since $t_i\in X(r)$ for every $i$, we may choose $z_i\in Z(m)$ such that $t_i=z_i+y_i$ where $y_i\in \F^{\{1,3\}}$ has small partition rank. Now $\sum_{i\leq \alpha} s_i\otimes t_i\in \F^{n_2}\otimes Z(m)+\sum_{i\leq \alpha} s_i\otimes y_i$. Moreover, as $\alpha$ is small and each $y_i$ has small partition rank, we have $\sum_{i\leq \alpha} s_i\otimes y_i\in L'_{\{1,2\}}(r)\otimes \F^{n_3}$ for some small $L'_{\{1,2\}}(r)\s \F^{\{1,2\}}$. So we have proved (\ref{eqnr4}) with $V_{\{2\}}=U^{\perp}$ and $V_{\{1,3\}}=Z(m)$.

\subsubsection{Why can we find such a $Q_{\{2\}}$ inside $2^{3^{d+3}}\cB'-2^{3^{d+3}}\cB'$?} \label{subsubhowtofind}

Now we describe why there must exist $Q_{\{2\}}$ with elements chosen from $2^{3^{3+3}}\cB'-2^{3^{3+3}}\cB'$ and having the required properties. We remark that in the general case this is done in Lemma \ref{claim}. We want to find a subspace $U\s \F^{n_2}$ of low codimension, and forcing multisets $Q_u\s \F^{\{1,3\}}$ ($u\in U$) consisting of pure tensors such that for every $u\in U$, $u\otimes Q_u\s 2^{3^{3+3}}\cB'-2^{3^{3+3}}\cB'$. Let $\cD$ be the multiset $\{v\otimes w: v\in \F^{n_1}, w\in \F^{n_3}\}$. Notice that if some set $R$ is dense in $\cD$, then by the induction hypothesis we can find a forcing set in $2^{3^{2+3}}R-2^{3^{2+3}}R$ consisting of pure tensors. Therefore it is enough to find a low codimensional subspace $U$ and dense sets $R_u\s \cD$ (for every $u\in U$) such that $u\otimes R_u\s 32\cB'-32\cB'$. As $\cB'$ is dense in $\cB$, we have a dense subset $S\s \F^{n_2}$ and dense subsets $T_s\s \cD$ ($s\in S$) such that $s\otimes T_s\s \cB'$ for every $s\in S$. By Bogolyubov's lemma (Lemma \ref{lemmabogolyubov}), there is a low codimensional subspace $U$ contained in $2S-2S$. To establish the existence of a dense $R_u\s \cD$ with $u\otimes R_u\s 32\cB'-32\cB'$ for every $u\in U$, it is enough to prove the following lemma.

\begin{lemma} \label{sumdense}
	Let $T_1,T_2,T_3,T_4$ be dense subsets of $\cD$. Then $\cD\cap \bigcap_{i\leq 4}(8T_i-8T_i)$ is dense in $\cD$.
\end{lemma}

Indeed, once we have this lemma, it follows that for any $s_1,s_2,s_3,s_4\in S$, the set $\cD\cap \bigcap_{i\leq 4} (8T_{s_i}-8T_{s_i})$ is dense in $\cD$. But if $u\in U$, then we can write $u=s_1+s_2-s_3-s_4$ for some $s_i\in S$, and then $u\otimes \bigcap_{i\leq 4} (8T_{s_i}-8T_{s_i})\s s_1\otimes \bigcap_{i\leq 4} (8T_{s_i}-8T_{s_i})+ s_2\otimes \bigcap_{i\leq 4} (8T_{s_i}-8T_{s_i})-s_3\otimes \bigcap_{i\leq 4} (8T_{s_i}-8T_{s_i})-s_4\otimes \bigcap_{i\leq 4} (8T_{s_i}-8T_{s_i})\s 32\cB'-32\cB'$.

Lemma \ref{sumdense} follows easily from the next two lemmas.

\begin{lemma} \label{find2dimsystem}
	Let $A$ be a dense subset of $\cD$. Then there exist a dense subspace $V\s \F^{n_1}$ and for each $v\in V$ a dense subspace $W_v\s \F^{n_3}$ such that $v\otimes W_v\s 8A-8A$ for every $v\in V$.
\end{lemma}

\begin{proof}
	There exist a dense subset $B\s \F^{n_1}$ and dense subsets $C_b\s \F^{n_3}$ for each $b\in B$ such that $b\otimes C_b\s A$. By Bogolyubov's lemma, $2B-2B$ contains a dense subspace $V\s \F^{n_1}$, and for every $b\in B$, $2C_b-2C_b$ contains a dense subspace $L_b \s \F^{n_3}$. For any $v\in V$, choose $b_1,b_2,b_3,b_4\in B$ with $v=b_1+b_2-b_3-b_4$ and set $W_v=\bigcap_{i\leq 4}L_{b_i}$. Note that $b_i\otimes w\in 2A-2A$ for every $i\leq 4$ and $w\in W_v$, therefore $v\otimes w\in 8A-8A$.
\end{proof}

\begin{lemma} \label{intersect2dimsystems}
	Suppose that we have dense subspaces $V,V'\s \F^{n_1}$, for each $v\in V$ a dense subspace $W_v\s \F^{n_3}$, and for each $v'\in V'$ a dense subspace $W'_{v'}\s \F^{n_3}$. Then $(\bigcup_{v\in V} v\otimes W_v)\cap (\bigcup_{v'\in V'} v'\otimes W'_{v'})=\bigcup_{v\in V\cap V'} v\otimes (W_v\cap W'_v)$. In particular, this intersection is a dense subset of $\cD$.
\end{lemma}

\begin{proof}
	The identity is trivial. Since the subspaces $V\cap V'$ and $W_v\cap W'_v$ are dense, the second assertion follows.
\end{proof}

\subsubsection{How can this be extended to $d>3$?} \label{subsubhowextend}

Now we briefly sketch what the main difficulties are in the $d>3$ case and how we can address them. The underlying strategy is similar: we take an ordering $\prec$ of the set of non-empty subsets $I\s \lbrack d-1\rbrack$, and for each such $I$ we choose $Q_I$ such that any array
\begin{equation}
r\in W_{\lbrack d\rbrack}+ \sum_{J\prec I} (W_J\otimes \F^{J^c}+\F^{J}\otimes W_{J^c})+\sum_{J\succeq I} \F^{J}\otimes H_{J^c}(r) \label{eqndlarge}
\end{equation}
with $r.q=0$ for almost all $q\in Q_I$ has $$r\in W_{\lbrack d \rbrack}+\sum_{J\preceq I} (U_J\otimes \F^{J^c}+\F^{J}\otimes U_{J^c})+\sum_{J\succ I} \F^{J}\otimes K_{J^c}(r)$$ where $U_J,U_{J^c},K_{J^c}(r)$ can have dimension slightly larger than those of $W_{J},W_{J^c}$ and $H_{J^c}$, but they are still low dimensional. In the $d=3$ case, we have made use of a decomposition $r=r_2+r_3+r_4$ where $r_4\in \F^{I}\otimes H_{I^c}(r)$, $r_2u$ has small partition rank and $r_3u$ is in a small subspace independent of $r$ for every $u\in \F^{I}$. In general, such a decomposition need not exist. For example, when $d=4$ and $I=\{1,2\}$, then an array in $W_{\{1\}}\otimes \F^{\{2,3,4\}}$ (or in $\F^{n_1}\otimes H_{\{2,3,4\}}(r)$ if we were to take $\{1,2\}\prec \{1\}$), when multiplied by some pure tensor $u\in \F^{\{1,2\}}$, yields a tensor which need not have small partition rank and need not lie a small space independent of $r$. However, by restricting the possible choices for $u$, we can make sure that the product is always zero. So we will take a decomposition $r=r_1+r_2+r_3+r_4$ such that $r_4\in \F^{I}\otimes H_{I^c}(r)$; for every pure tensor $u\in \F^{I}$, $r_2u$ has small partition rank and $r_3u$ lies in a small space depending only on $u$; and crucially, for every $q\in Q_I$, $r_1.q=0$. To achieve this, we need to insist that $J\prec I$ whenever $J\subsetneq I$ and that $Q_I$ is orthogonal to certain subspaces. To see this, note that in the above example where $d=4$ and $I=\{1,2\}$ we need that $\{1\}\prec \{1,2\}$ and $Q_{\{1,2\}}$ is orthogonal to $W_{\{1\}}\otimes \F^{\{2,3,4\}}$. (If we had $\{1,2\}\prec \{1\}$, then in (\ref{eqndlarge}) we would have a term $\F^{n_1}\otimes H_{\{2,3,4\}}(r)$ rather than $W_{\{1\}}\otimes \F^{\{2,3,4\}}$, which we could not control.)

We also need to generalise Lemma \ref{sumdense} to the case $d>3$. Instead of using $\bigcup_{v\in V}v\otimes W_v$ as in Lemma \ref{find2dimsystem}, we need to define an object in $\cB$ such that
\begin{enumerate}
	\item an instance of the object can be found in $k\B'-k\B'$ for some small $k$ whenever $\cB'$ is dense in $\cB$ (generalising Lemma \ref{find2dimsystem})
	
	\item the intersection of few instances of this object is a dense subset of $\cB$ (generalising Lemma \ref{intersect2dimsystems})
\end{enumerate}

In the next subsection we describe this object and show that it has the required properties.

\subsection{Construction of some auxiliary sets}


\begin{definition} \label{ksystem}
	Suppose that we have a collection of vector spaces as follows. The first one is $U\s \F^{n_1}$, of codimension at most $l$. Then, for every $u_1\in U$, there is some $U_{u_1}\s \F^{n_2}$. In general, for every $2\leq k\leq d$ and every $u_1\in U,u_2\in U_{u_1},\dots,u_{k-1}\in U_{u_1,\dots,u_{k-2}}$, there is a subspace $U_{u_1,\dots,u_{k-1}}\s \F^{n_k}$. Assume, in addition, that the codimension of $U_{u_1,\dots,u_{k-1}}$ in $\F^{n_k}$ is at most $l$ for every $u_1\in U,\dots,u_{k-1}\in U_{u_1,\dots,u_{k-2}}$. Then the multiset $Q=\{u_1\otimes \dots \otimes u_d: u_1\in U,\dots,u_d\in U_{u_1,\dots,u_{d-1}}\}$ is called an $l$-system.
	
\end{definition}

The next lemma is the generalisation of Lemma \ref{intersect2dimsystems} from the previous subsection.

\begin{lemma} \label{intersect}
	Let $Q$ be an $l$-system and let $Q'$ be an $l'$-system. Then $Q\cap Q'$ contains an $(l+l')$-system.
\end{lemma}

\begin{proof}
	Let $Q$ have spaces as in Definition \ref{ksystem} and let $Q'$ have spaces $U'_{u'_1,\dots,u'_{k-1}}$.
	We define an $(l+l')$-system $P$ contained in $Q\cap Q'$ as follows. Let $V=U\cap U'$. Suppose we have defined $V_{v_1,\dots,v_{j-1}}$ for all $j\leq k$. Let $v_1\in V,v_2\in V_{v_1},\dots,v_{k-1}\in V_{v_1,\dots,v_{k-2}}$.  We let $V_{v_1\dots,v_{k-1}}=U_{v_1\dots,v_{k-1}}\cap U'_{v_1\dots,v_{k-1}}$. This is well-defined and has codimension at most $l+l'$ in $\F^{n_k}$. Let $P$ be the $(l+l')$-system with spaces $V_{v_1,\dots,v_{k-1}}$.
\end{proof}

The next lemma is the generalisation of Lemma \ref{find2dimsystem} from the previous subsection.

\begin{lemma} \label{findsystem}
	Let $\cB'\subset \cB$ be a multiset such that $|\cB'|\geq \delta |\cB|$. Then there exists an $f_1$-system whose elements are chosen from $f_2\cB'-f_2\cB'$ with $f_1=C\cdot4^{d}(\log(2^d/\d))^4$ and $f_2=4^d$.
\end{lemma}

\begin{proof}
	The proof is by induction on $d$. The case $d=1$ is a direct consequence of Lemma \ref{lemmabogolyubov}. Suppose that the lemma has been proved for all $d'<d$ and let $\cB'\subset \cB$ be a multiset such that $|\cB'|\geq \delta |\cB|$. Let $\cD$ be the multiset $\{v_2\otimes\dots\otimes v_d: v_2\in \F^{n_2},\dots,v_d\in \F^{n_d}\}$. For each $u\in \F^{n_1}$, let $\cB_u'=\{s\in \cD: u\otimes s\in \cB'\}$ and let $T=\{u\in \F^{n_1}: |\cB_u'|\geq\frac{\delta}{2}|\cD|\}$. By averaging, we have that $|T|\geq\frac{\delta}{2}|\F^{n_1}|$. Now by the induction hypothesis, for every $t\in T$, there exists a $g_1$-system in $\F^{n_2}\otimes \dots \otimes \F^{n_d}$ (whose definition is analogous to the definition of a system in $\F^{n_1}\otimes \dots \otimes \F^{n_d}$), called $P_t$, contained in $g_2\cB'_t-g_2\cB'_t$ where $g_1=C\cdot 4^{d-1}(\log(2^{d}/\d))^4$ and $g_2=4^{d-1}$. By Lemma \ref{lemmabogolyubov}, $2T-2T$ contains a subspace $U\s \F^{n_1}$ of codimension at most $C(\log(2/\d))^4$. For each $u\in U$, write $u=t_1+t_2-t_3-t_4$ arbitrarily with $t_i\in T$, and let $Q_u=P_{t_1}\cap P_{t_2}\cap P_{t_3}\cap P_{t_4}$, which is a $g_3$-system with $g_3=4g_1=C\cdot 4^d(\log (2^d/\d))^4$, by Lemma \ref{intersect}. Thus, $Q=\bigcup_{u\in U} (u\otimes Q_u)$ is indeed an $f_1$-system. Moreover, for any $u\in U,s\in Q_u$, we have $u\otimes s=t_1\otimes s+t_2\otimes s-t_3\otimes s-t_4\otimes s$ for some $t_i\in T$ and $s\in \bigcap_{i\leq 4} P_{t_i}$. Then $t_i\otimes s\in g_2\cB'-g_2\cB'$, therefore $u\otimes s\in 4g_2\cB'-4g_2\cB'$, so the elements of $Q$ are indeed chosen from $f_2\cB'-f_2\cB'$.
\end{proof}

The next lemma describes a property of systems which was not needed for us in the $d=3$ case, but is crucial in the general case. It is required for finding a suitable decomposition $r=r_1+r_2+r_3+r_4$ described at the end of the previous subsection. Indeed, we need a set $Q_I$ which is orthogonal to certain spaces of the form $W_J\otimes \F^{J^c}$ (ie. is contained in $W_J^{\perp}\otimes \F^{J^c}$) to make sure that $r_1.q=0$ for every $q\in Q_I$. We will use the following lemma to guarantee the existence of such a set $Q_I$.

\begin{lemma} \label{systemspace}
	Let $Q$ be a $k$-system and for every non-empty $I\s \lbrack d\rbrack$, let $L_I\s \F^{I}$ be a subspace of codimension at most $l$. Let $T=\bigcap_I(L_I\otimes \F^{I^c})$. Then $Q\cap T$ contains an $f$-system for $f=k+2^dl$.
\end{lemma}

\begin{proof}
	Let the spaces of $Q$ be $U_{u_1,\dots,u_{j-1}}$. It suffices to prove that for every $1\leq j\leq d$, and every $u_1\in U,\dots,u_{j-1}\in U_{u_1,\dots,u_{j-2}}$, the codimension of $(u_1\otimes \dots\otimes u_{j-1}\otimes U_{u_1,\dots,u_{j-1}})\cap \bigcap_{I\subset \lbrack j\rbrack, j\in I}(L_I\otimes \F^{\lbrack j\rbrack \setminus I})$ in $u_1\otimes \dots\otimes u_{j-1}\otimes U_{u_1,\dots,u_{j-1}}$ is at most $2^dl$. Thus, it suffices to prove that for every $I\s \lbrack j\rbrack$ with $j\in I$, the codimension of $(u_1\otimes \dots\otimes u_{j-1}\otimes U_{u_1,\dots,u_{j-1}})\cap(L_I\otimes \F^{\lbrack j\rbrack \setminus I})$ in $u_1\otimes \dots\otimes u_{j-1}\otimes U_{u_1,\dots,u_{j-1}}$ is at most $l$. But this is equivalent to the statement that $\big((\bigotimes_{i\in I\setminus \{j\}}u_i)\otimes U_{u_1,\dots,u_{j-1}}\big)\cap L_I$ has codimension at most $l$ in $(\bigotimes_{i\in I\setminus \{j\}}u_i)\otimes U_{u_1,\dots,u_{j-1}}$, which clearly holds.
\end{proof}

\subsection{The proof of Lemma \ref{mainlemma}}

We now turn to the proof of Lemma \ref{mainlemma}. As described in the outline, the first step is to find a $Q_{\lbrack d\rbrack}$ such that if $r.q=0$ for almost all $q\in Q_{\lbrack d\rbrack}$, then $r=x+y$ where $x\in V_{\lbrack d\rbrack}$ for a small space $V_{\lbrack d\rbrack}$ independent of $r$, and $y$ has low partition rank.

\begin{lemma} \label{focusondeg}
	Let $d\geq 2$ and suppose that Lemma \ref{mainlemma} has been proved for $d'=d-1$. Let $\cB'\s \cB$ be such that $|\cB'|\geq \d |\cB|$ for some $\d>0$. Then there exist some $Q\subset 2\cB'-2\cB'$ consisting of pure tensors and a subspace $V_{\lbrack d\rbrack}\s \F^{\lbrack d\rbrack}$ of dimension at most $4C(\log (2/\d))^4$ with the following property. Any array $r$ with $r.q=0$ for at least $\frac{7}{8}|Q|$ choices $q\in Q$ can be written as $r=x+y$ where $x\in V_{\lbrack d \rbrack}$ and $y$ is $f$-degenerate for $f=G(d-1,\frac{\d}{4|\F|^{4C(\log 2/\d)^4}})$.
\end{lemma}

\begin{proof}	
	Let $\cD$ be the multiset $\{u_1\otimes \dots \otimes u_{d-1}:u_1\in \F^{n_1},\dots,u_{d-1}\in\F^{n_{d-1}}\}$ and let $\cD'=\{t\in \cD: t\otimes u\in \cB' \text{ for at least } \frac{\d}{2}|\F|^{n_d} \text{ choices } u\in \F^{n_d}\}$. Clearly, we have $|\cD'|\geq \frac{\d}{2}|\cD|$. Moreover, by Lemma \ref{lemmabogolyubov}, for every $t\in \cD'$, there exists a subspace $U_t\s \F^{n_d}$ of codimension at most $C(\log (2/\d))^4$ such that $t\otimes U_t\s 2\cB'-2\cB'$. After passing to suitable subspaces, we may assume that all $U_t$ have the same codimension $k\leq C(\log (2/\d))^4$. Now let $Q=\cup_{t\in \cD'}(t\otimes U_t)$.
	
	Write $R$ for the set of arrays $r$ with $r.q=0$ for at least $\frac{7}{8}|Q|$ choices $q\in Q$.
	
	We now define a sequence of subspaces $0=V(0)\s V(1) \s \dots \s V(m)\s \F^{\lbrack d \rbrack}$ recursively as follows.
	
	Given $V(j)$, if for every $r\in R$ there are at least $\frac{|\cD'|}{2}$ choices $t\in \cD'$ with $rt\in V(j)t$, then we set $m=j$ and terminate. (Here and below, for a subspace $L\s \cG$ and an array $s\in \F^{I}$, we write $Ls$ for the subspace $\{rs: r\in L\}\s \F^{I^c}$.)
	
	Else, we choose some $r\in R$ such that there are at most $\frac{|\cD'|}{2}$ choices $t\in \cD'$ with $rt\in V(j)t$. We set $V(j+1)=V(j)+\text{span}(r)$. Note that $r.(t\otimes s)=(rt).s$ for every $s\in U_t$. If $rt\not \in U_t^{\perp}$, then $(rt).s=0$ holds for only a proportion $1/|\F|\leq 1/2$ of all $s\in U_t$. Thus, as $r\in R$, we have $rt\in U_t^{\perp}$ for at least $\frac{3}{4}|\cD'|$ choices $t\in \cD'$. Moreover, since $rt\in V(j)t$ holds for at most $\frac{|\cD'|}{2}$ choices $t\in \cD'$, it follows that for at least $\frac{|\cD'|}{4}$ choices $t\in \cD'$ we have $rt\in U_t^{\perp}\setminus V(j)t$. Thus, we have $\dim (U_t^{\perp}\cap V(j+1)t)>\dim (U_t^{\perp}\cap V(j)t)$ for at least $\frac{|\cD'|}{4}$ choices $t\in \cD'$.
	
	However, for any $j$ we have $\sum_{t\in \cD'} \dim (U_t^{\perp}\cap V(j)t)\leq \sum_{t\in \cD'} \dim U_t^{\perp}\leq C|\cD'|(\log (2/\d))^4$. Thus, we get $m\leq 4C(\log (2/\d))^4$. Set $V_{\lbrack d\rbrack}=V(m)$. Then $\dim V_{\lbrack d\rbrack}\leq 4C(\log (2/\d))^4$, as claimed.
	
	Now let $r\in R$ be arbitrary. By definition, there are at least $|\cD'|/2$ choices $t\in \cD'$ with $rt\in V_{\lbrack d\rbrack}t$. Then there is some $v\in V_{\lbrack d\rbrack}$ such that $rt=vt$ for at least $\frac{|\cD'|}{2|V_{\lbrack d\rbrack}|}$ choices $t\in \cD'$, and hence also for at least $\frac{\d|\cD|}{4|V_{\lbrack d\rbrack}|}$ choices $t\in \cD$. Note that $\frac{\d}{4|V_{\lbrack d\rbrack}|}\geq \frac{\d}{4|\F|^{4C(\log  2/\d)^4}}$, therefore by Lemma \ref{biastorank}, $r-v$ is $f$-degenerate.
\end{proof}

\begin{definition}
	Let $k$ be a positive integer and let $0\leq \a\leq 1$. Let $Q$ be a multiset with elements chosen from $\cG$ (with arbitrary multiplicity). We say that $Q$ is $(k,\a)$-forcing if the set of all arrays $r\in \cG$ with $r.q=0$ for at least $\a|Q|$ choices $q\in Q$ is contained in a set of the from $\sum_{I\subset \lbrack d\rbrack,I\neq \emptyset} V_I\otimes \F^{I^c}$ for some $V_I\s \F^{I}$ of dimension at most $k$.
\end{definition}

We now turn to the main part of the proof of Lemma \ref{mainlemma}. For each non-empty $I\s \lbrack d-1 \rbrack$ we will construct $Q_I$ as defined in the next result, and (roughly) we will take $Q=Q_{\lbrack d\rbrack}\cup \bigcup_{I\s \lbrack d-1 \rbrack,I\neq \emptyset} Q_I$, where $Q_{\lbrack d\rbrack}$ is provided by Lemma \ref{focusondeg}. The properties that $Q_I$ has are generalisations of the properties that $Q_{\{2\}}$ had in Subsection \ref{subsectionoverviewranklemma}. Accordingly, the next lemma is the generalisation of the discussion in Subsubsection \ref{subsubhowtofind}.

\begin{lemma} \label{claim}
	Let $d\geq 2$ and suppose that Lemma \ref{mainlemma} has been proved for every $d'<d$. Let $\cB'\s \cB$ have $|\cB'|\geq \d|\cB|$ for some $0<\d\leq 1/2$. Let $k\geq G(d-1,\d)$ be arbitrary, let $I\s \lbrack d-1\rbrack, I\neq \emptyset$, and let $W_J\s \F^{J}$ be subspaces of dimension at most $k$ for every $J\s I,J\neq I,J\neq \emptyset$. Then there exist a multiset $Q'$, and a multiset $Q_s$ for each $s\in Q'$ with the following properties.
	
	\begin{enumerate}[label=(\arabic*)]
		\item The elements of $Q'$ are pure tensors chosen from $\bigcap_{J\subset I,J\neq I,J\neq \emptyset}(W_J^{\perp}\otimes \F^{I \setminus J})\s \F^{I}$
		
		\item $Q'$ is $(f_1,1-f_2)$-forcing with $f_1=G(|I|,|\F|^{-2^{d+1}dk})$, $f_2=2^{-3^{d+2}}$
		
		\item For each $s\in Q'$, the elements of $Q_s$ are pure tensors chosen from $\F^{I^c}$
		
		\item For each $s\in Q'$, $Q_s$ is $(f_3,1-f_4)$-forcing with $f_3=G(d-|I|,|\F|^{-2^{3^{d+4}}C(\log (2^{d-1}/\d))^4})$, $f_4=2^{-3^{d+2}}$
		
		\item $\max_{s\in Q'}|Q_s|\leq 2\min_{s\in Q'}|Q_s|$
		
		\item The elements of the multiset $Q_I:=\{s\otimes t: s\in Q',t\in Q_s\}=\bigcup_{s\in Q'}(s\otimes Q_s)$ are chosen from $f_5\cB'-f_5\cB'$ with $f_5=2^{3^{d+3}}$.
	\end{enumerate}
\end{lemma}

\begin{proof}
	By symmetry, we may assume that $I=\lbrack a\rbrack$ for some $1\leq a \leq d-1$. Let $\cC$ be the multiset $\{u_1\otimes \dots \otimes u_a: u_i\in \F^{n_i}\}$ and let $\cD$ be the multiset $\{u_{a+1}\otimes \dots \otimes u_d: u_i\in \F^{n_i}\}$. For each $s\in \cC$, let $\cD_s=\{t\in \cD: s\otimes t\in \cB'\}$. Also, let $\cC'=\{s\in \cC: |\cD_s|\geq \frac{\d}{2}|\cD|\}$. Clearly, $|\cC'|\geq \frac{\d}{2}|\cC|$. By Lemma \ref{findsystem}, there exists a $g_1$-system $R$ (with respect to $\F^{I}$) with elements chosen from $g_2\cC'-g_2\cC'$ with $g_1=C\cdot 4^d(\log(2^{d-1}/\d))^4$ and $g_2=4^d$. By Lemma \ref{systemspace}, $R\cap \bigcap_{J\subset I,J\neq I,J\neq \emptyset} (W_J^{\perp}\otimes \F^{I\setminus J})$ contains a $g_3$-system $T'$ for $g_3=C\cdot 4^d(\log(2^{d-1}/\d))^4+2^dk$. Now $|T'|\geq |\F|^{-dg_3}|\cC|$. By Lemma \ref{mainlemma} (applied to $a$ in place of $d$), it follows that there exists a multiset $Q'$ whose elements are pure tensors chosen from $g_4T'-g_4T'$ and which is $(g_5,1-g_6)$-forcing for $g_4=2^{3^{a+3}}\leq 2^{3^{d+2}}$, $g_5=G(a,|\F|^{-dg_3})$ and $g_6=2^{-3^{a+3}}\geq 2^{-3^{d+2}}$. Note that since $\d\leq 1/2$, we have $C\cdot 4^d(\log (2^{d-1}/\d))^4= C\cdot 4^d (d-1+\log (1/\d))^4\leq C\cdot 4^d (d\log (1/\d))^4$. But this is at most as $G(d-1,\d)\leq k$, so $g_3\leq 2\cdot 2^dk$, therefore $Q'$ satisfies (1) and (2) in the statement of this lemma.
	
	By Lemma \ref{findsystem}, for each $s\in \cC'$ there exists a $g_7$-system $R_s$ (with respect to $\F^{I^c}$) contained in $g_8\cD_s-g_8\cD_s$, where $g_7=C\cdot 4^d(\log (2^{d-1}/\d))^4$ and $g_8=4^d$. For every $s\in Q'$, choose $s_1,\dots,s_{l+l'}\in \cC'$ with $l,l'\leq 2^{3^{d+3}}$ such that $s=s_1+\dots+s_l-s_{l+1}-\dots-s_{l+l'}$ (this is possible, since the elements of $Q'$ are chosen from $2g_2g_4\cC'-2g_2g_4\cC'$ and $2g_2g_4\leq 2^{3^{d+3}}$), and let $P_s=\bigcap_{i\leq l+l'}R_s$. By Lemma \ref{intersect}, $P_s$ contains a $g_9$-system with $g_9=2\cdot 2^{3^{d+3}}\cdot C \cdot 4^d(\log (2^{d-1}/\d))^4$, therefore $|P_s|\geq g_{10}|\cD|$ for $g_{10}=|\F|^{-dg_9}\geq |\F|^{-2^{3^{d+4}}C(\log(2^{d-1}/\d))^4}$. By Lemma \ref{mainlemma} (applied to $d-a$ in place of $d$), for every $s\in Q'$ there exists a multiset $Q_s$ consisting of pure tensors with elements chosen from $g_{11}P_s-g_{11}P_s$ which is $(g_{12},1-g_{13})$-forcing for $g_{11}=2^{3^{d-a+3}}\leq 2^{3^{d+2}}$, $g_{12}=G(d-a,|\F|^{-dg_9})\leq G(d-a,|\F|^{-2^{3^{d+4}}C(\log (2^{d-1}/\d))^4})$ and $g_{13}=2^{-3^{d-a+3}}\geq 2^{-3^{d+2}}$. Notice that if we repeat every element of $Q_s$ the same number of times, then the multiset obtained is still $(g_{12},1-g_{13})$-forcing, so we may assume that $\max_{s\in Q'}|Q_s|\leq 2\min_{s\in Q'}|Q_s|$. Thus, the $Q_s$ satisfy (3), (4) and (5).
	
	Define $Q_I=\{s\otimes t:s\in Q',t\in Q_s\}=\bigcup_{s\in Q'}(s\otimes Q_s)$. Note that as $R_s\subset g_8\cD_s-g_8\cD_s$ for all $s\in \cC'$, we have $s\otimes R_s\subset g_8\cB'-g_8\cB'$ for all $s\in \cC'$. But the elements of $Q'$ are chosen from $2g_2g_4\cC'-2g_2g_4\cC'$, so $s\otimes P_s\s 4g_2g_4g_8\cB'-4g_2g_4g_8\cB'$ for all $s\in Q'$. Finally, the elements of $Q_s$ are chosen from $g_{11}P_s-g_{11}P_s$, so the elements of $s\otimes Q_s$ are chosen from $8g_2g_4g_8g_{11}\cB'-8g_2g_4g_8g_{11}\cB'$ for every $s\in Q'$. Since $8g_2g_4g_8g_{11}\leq 8\cdot (4^d)^2\cdot (2^{3^{d+2}})^2=2^{3+4d+2\cdot 3^{d+2}}\leq 2^{3^{d+3}}$, property (6) is satisfied.
\end{proof}

The next lemma is the last ingredient of the proof. It is a generalisation of the discussion in Subsubsection \ref{subsubwhyworks}. Given a tensor $r\in V_{\lbrack d\rbrack}+\sum_{I\s \lbrack d-1\rbrack, I\neq \emptyset} \F^{I}\otimes H_{I^c}(r)$, we turn the terms $\F^{I}\otimes H_{I^c}(r)$ one by one into terms $V_I\otimes \F^{I^c}+\F^{I}\otimes V_{I^c}$ where $V_J$ are small and do not depend on $r$. (Note that this is not quite the same as our approach to the case $d=3$.) As briefly explained in Subsubsection \ref{subsubhowextend}, the order in which the various $I$ are considered is important: we define $\prec$ to be any total order on the set of non-empty subsets of $\lbrack d-1 \rbrack$ such that if $J\subsetneq I$ then $J\prec I$. It is worth noting that unlike in the $d=3$ case, the subspaces $V_J,V_{J^c}$ with $J\prec I$ are allowed to change when $V_I$ and $V_{I^c}$ get defined (although in fact the $V_{J^c}$ will not change, and the $V_J$ change only for $J\subsetneq I$). All we require is that they do not become much larger.

\begin{lemma} \label{keylemma}
	Let $d\geq 2$, $0<\d\leq 1/2$ and $k\geq G(d-1,\d)^{2}$. Let $I\s \lbrack d-1\rbrack, I\neq \emptyset$ and let $W_J\s \F^{J},W_{J^c}\s \F^{J^c}$ be subspaces of dimension at most $k$ for every $J\prec I$. Moreover, let $W_{\lbrack d\rbrack}\s \F^{\lbrack d\rbrack}$ have dimension at most $k$. Suppose that $Q',Q_s$ (and $Q_I$) have the six properties described in Lemma \ref{claim}. Then any array $$r\in W_{\lbrack d\rbrack}+ \sum_{J\prec I} (W_J\otimes \F^{J^c}+\F^{J}\otimes W_{J^c})+\sum_{J\succeq I} \F^{J}\otimes H_{J^c}(r)$$ with $\dim(H_{J^c}(r))\leq k$ and the property that $r.q=0$ for at least $(1-\frac{1}{4}(2^{-3^{d+2}})^2)|Q_I|$ choices $q\in Q_I$ is contained in $$W_{\lbrack d \rbrack}+\sum_{J\preceq I} (U_J\otimes \F^{J^c}+\F^{J}\otimes U_{J^c})+\sum_{J\succ I} \F^{J}\otimes K_{J^c}(r)$$ for some $U_J\s \F^{J},U_{J^c}\s \F^{J^c}$ not depending on $r$ and some $K_{J^c}(r)\s \F^{J^c}$ possibly depending on $r$, all of dimension at most $k^{2c_2(|I|)}$.
\end{lemma}

\begin{proof}
	By (4) in Lemma \ref{claim}, for every $s\in Q'$ there exist subspaces $V_J(s)\s \F^{J}$ for every $J\s I^c,J\neq \emptyset$, with dimension at most
	$g_1=G(d-1,|\F|^{-2^{3^{d+4}}C(\log 2^{d-1}/\d)^4})$
	such that the set of arrays $t\in \F^{I^c}$ with $t.q=0$ for at least $(1-g_2)|Q_s|$ choices $q\in Q_s$ is contained in $\sum_{J\s I^c,J\neq \emptyset} V_J(s)\otimes \F^{I^c\setminus J}$, where $g_2=2^{-3^{d+2}}$. Note, for future reference, that
	\begin{align*}
	g_1&=G(d-1,|\F|^{-2^{3^{d+4}}C(\log 2^{d-1}/\d)^4})= ((\log |\F|)^2 c_1(d-1)2^{3^{d+4}}C(\log 2^{d-1}/\d)^4)^{c_2(d-1)} \\
	&\leq ((\log |\F|)^2 c_1(d-1)2^{3^{d+4}}C(d\log 1/\d)^4)^{c_2(d-1)}\leq ((\log |\F|)^2 (c_1(d-1))^2(\log 1/\d)^4)^{c_2(d-1)} \\
	&\leq G(d-1,\d)^{4}\leq k^2.
	\end{align*}
	
	Let $R$ consist of the set of arrays with $r\in W_{\lbrack d\rbrack}+ \sum_{J\prec I} (W_J\otimes \F^{J^c}+\F^{J}\otimes W_{J^c})+\sum_{J\succeq I} \F^{J}\otimes H_{J^c}(r)$ with $\dim(H_{J^c}(r))\leq k$ and the property that $r.q=0$ for at least $(1-\frac{1}{4}(2^{-3^{d+2}})^2)|Q_I|$ choices $q\in Q_I$.
	
	Let $r\in R$. Then by averaging and using (5) from Lemma \ref{claim}, for at least $(1-g_3)|Q'|$ choices $s\in Q'$ we have $r.(s\otimes t)=0$ for at least $(1-g_2)|Q_s|$ choices $t\in Q_s$, where $g_3=\frac{1}{2}2^{-3^{d+2}}$. Thus, (noting that $r.(s\otimes t)=(rs).t$), $rs\in \sum_{J\s I^c,J\neq \emptyset} V_J(s)\otimes \F^{I^c\setminus J}$ holds for at least $(1-g_3)|Q'|$ choices $s\in Q'$. Let $Q'(r)$ be the submultiset of $Q'$ consisting of those $s\in Q'$ for which $rs\in \sum_{J\s I^c,J\neq \emptyset} V_J(s)\otimes \F^{I^c\setminus J}$. Then we have $|Q'(r)|\geq (1-g_3)|Q'|$.
	
	\smallskip
	
	Note that we can write $r=r_1+r_2+r_3+r_4$ where $$r_1\in \sum_{J\subset I,J\neq I,J\neq \emptyset} W_J\otimes \F^{J^c},$$ $$r_2\in \sum_{J\prec I, J\not \subset I} (W_J\otimes \F^{J^c}+\F^{J}\otimes W_{J^c})+\sum_{J\succ I} \F^{J}\otimes H_{J^c}(r),$$ $$r_3\in W_{\lbrack d\rbrack}+\sum_{J\subset I,J\neq I,J\neq \emptyset} \F^{J}\otimes W_{J^c},$$ $$r_4\in \F^{I}\otimes H_{I^c}(r).$$
	By (1) in Lemma \ref{claim}, the elements of $Q'$ belong to $\bigcap_{J\subset I,J\neq I,J\neq \emptyset}(W_J^{\perp}\otimes \F^{I \setminus J})$, so we have $r_1s=0$ for every $s\in Q'$.
	
	Note that for every pure tensor $s\in \F^{I}$, $r_2s$ is $2^{d}k$-degenerate. Indeed, for any $J\s \lbrack d-1 \rbrack$ with $J\not \s I$ there are some $s_1\in \F^{I\cap J},s_2\in \F^{I\cap J^c}$ with $s=s_1\otimes s_2$. Then $(W_J\otimes \F^{J^c})s\s (W_Js_1)\otimes \F^{I^{c}\setminus J}$. Since $\dim(W_Js_1)\leq k$, $J\not \s I$ and $d\in I^c\setminus J$, any tensor in $(W_Js_1)\otimes \F^{I^{c}\setminus J}$ is $k$-degenerate. Similarly, any tensor in $(\F^{J}\otimes W_{J^c})s$ or $(\F^{J}\otimes H_{J^c}(r))s$ is also $k$-degenerate, so $r_2s$ is indeed $2^dk$-degenerate. Since $Q'$ consists of pure tensors, this holds for every $s\in Q'$.
	
	
	Also, $r_3s\in \sum_{J\subset I,J\neq I} ((\F^{J}\otimes W_{J^c})s)$. It follows that for every $s\in Q'(r)$, there exists some $t(s)\in V_{I^c}(s)+\sum_{J\subset I,J\neq I} ((\F^{J}\otimes W_{J^c})s)$ such that $r_4s-t(s)$ is $g_4$-degenerate for $g_4=g_1+2^dk$ (we have used that $\dim (V_{J}(s))\leq g_1$). To ease the notation, write $T(s)$ for the space $V_{I^c}(s)+\sum_{J\subset I,J\neq I} ((\F^{J}\otimes W_{J^c})s)$. We claim that the dimension of $T(s)$ is at most $g_4=g_1+2^dk$. Indeed, $\dim(V_{I^c})\leq g_1$, so it suffices to prove that $\dim((\F^{J}\otimes W_{J^c})s)\leq k$ for every $J\s I,J\neq I$. Since $s\in Q'$, $s$ is a pure tensor, so for any such $J$ we have $s=s_1\otimes s_2$ for some $s_1\in \F^{J},s_2\in \F^{I\setminus J}$. But then $(\F^{J}\otimes W_{J^c})s\s W_{J^c}s_2$, which has dimension at most $\dim(W_{J^c})\leq k$.
	
	Let us define a sequence of subspaces $0=Z(0)\s Z(1)\s \dots \s Z(m)\s \F^{I^c}$ recursively as follows. Given $Z(j)$, if for all $r\in R$ we have that for all but at most $2g_3|Q'|$ choices $s\in Q'$ there is some $z\in Z(j)$ such that $r_4s-z$ is $(g_4+1)g_4$-degenerate, then set $m=j$ and terminate.
	
	Else, choose some $r\in R$ such that for at least $2g_3|Q'|$ choices $s\in Q'$ there is no $z\in Z(j)$ such that $r_4s-z$ is $(g_4+1)g_4$-degenerate, and set $Z(j+1)=Z(j)+H_{I^c}(r)$. Recall that for every $s\in Q'(r)$, and in particular, for at least $(1-g_3)|Q'|$ choices $s\in Q'$, there exists some $t(s)\in T(s)$ such that $r_4s-t(s)$ is $g_4$-degenerate. So for at least $g_3|Q'|$ choices $s\in Q'$ there is some $t(s)\in T(s)$ such that $r_4s-t(s)$ is $g_4$-degenerate, but there is no $z\in Z(j)$ such that $r_4s-z$ is $(g_4+1)g_4$-degenerate. In this case there is no $z\in Z(j)$ such that $z-t(s)$ is $g_4^2$-degenerate. On the other hand, since $r_4s\in H_{I^c}(r)\s Z(j+1)$, there is some $z\in Z(j+1)$ such that $z-t(s)$ is $g_4$-degenerate. For any $i$, let $K(i,s)$ be the subspace of $T(s)$ spanned by those $t\in T(s)$ for which there is some $z\in Z(i)$ with $z-t$ being $g_4$-degenerate. Since the dimension of $T(s)$ is at most $g_4$, we have $t(s)\not \in K(j,s)$, else there would exist some $z\in Z(j)$ such that $z-t(s)$ is $g_4^2$-degenerate. On the other hand, $t(s)\in K(j+1,s)$. Thus, $\dim K(j+1,s)>\dim K(j,s)$. This holds for at least $g_3|Q'|$ choices $s\in Q'$, so
	$$\sum_{s\in Q'} \dim K(j+1,s)\geq g_3|Q'|+\sum_{s\in Q'} \dim K(j,s).$$
	
	Since $K(m,s)\s T(s)$, we have $\dim K(m,s)\leq g_4$. Thus,
	$$|Q'|g_4\geq \sum_{s\in Q'} \dim K(m,s)\geq mg_3|Q'|,$$
	so $m\leq \frac{g_4}{g_3}$ and $\dim Z(m)\leq \frac{kg_4}{g_3}$. Write $Z=Z(m)$.
	
	Now let $r\in R$. Let $X(r)$ be the set consisting of those $x\in H_{I^c}(r)$ for which there is some $z\in Z$ with $x-z$ being $(g_4+1)g_4$-degenerate. Then $r_4s\in X(r)$ apart from at most $2g_3|Q'|$ choices $s\in Q'$. Let $t_1,\dots,t_{\a}$ be a maximal linearly independent subset of $X(r)$ and extend it to a basis $t_1,\dots,t_{\a},t'_1,\dots,t'_{\beta}$ for $H_{I^c}(r)$. Now if a linear combination of $t_1,\dots,t_{\a},t'_1,\dots,t'_{\beta}$ is in $X(r)$, then the coefficients of $t'_1,\dots,t'_{\beta}$ are all zero. Write $r_4=\sum_{i\leq \a}s_i\otimes t_i+\sum_{j\leq \beta}s'_j\otimes t'_j$ for some $s_i,s'_j\in \F^{I}$. Since $r_4q\in X(r)$ for at least $(1-2g_3)|Q'|=(1-2^{-3^{d+2}})|Q'|$ choices $q\in Q'$, we have, for all $j$, that $s'_j.q=0$ for at least $(1-2^{-3^{d+2}})|Q'|$ choices $q\in Q'$. Thus, by (2) in Lemma \ref{claim} there exist subspaces $L_J\s \F^{J}$ ($J\s I,J\neq \emptyset$) not depending on $r$, and of dimension at most $G(|I|,|\F|^{-2^{d+1}dk})$ such that $s'_j\in \sum_{J\s I,J\neq \emptyset} L_J\otimes \F^{I\setminus J}$ for all $j$. Thus, $r_4\in \sum_{i\leq \a}s_i\otimes t_i+\sum_{J\s I,J\neq \emptyset} L_J\otimes \F^{J^c}$. Moreover, for every $i\leq \a$, we have $t_i\in X(r)$, so there exist $z_i\in Z$ such that $t_i-z_i$ is $(g_4+1)g_4$-degenerate. It follows that $r_4\in \F^{I}\otimes Z+\sum_{J\supset I,J\neq I,J\subset \lbrack d-1\rbrack} \F^{J}\otimes K'_{J^c}(r)+\sum_{J\s I,J\neq \emptyset} L_J\otimes \F^{J^c}$ for some $K'_{J^c}(r)\s \F^{J^c}$ of dimension at most $\a\cdot (g_4+1)g_4\leq k\cdot (g_4+1)g_4$.

	We claim that $\dim(Z),\dim(K'_{J^c})$ and $\dim(L_J)$ are all bounded by $k^{2c_2(|I|)}-k$.
	
	Firstly, note that $g_4=g_1+2^dk\leq k^2+2^dk\leq 2k^2$.
	
	Now $\dim(K'_{J^c})\leq k(g_4+1)g_4\leq k^{6}\leq k^{2c_2(|I|)}-k$. Also, $\dim(Z)\leq \frac{kg_4}{g_3}\leq k^4\leq k^{2c_2(|I|)}-k$. Finally, 
	\begin{align*}
	\dim(L_J)&\leq G(|I|,|\F|^{-2^{d+1}dk})=((\log |\F|)^2c_1(|I|)(2^{d+1}dk))^{c_2(|I|)}\leq ((\log |\F|)^2c_1(d-1)^2k)^{c_2(|I|)} \\
	&\leq G(d-1,\d)^{2}k^{c_2(|I|)}\leq k^{c_2(|I|)+1}\leq k^{2c_2(|I|)}-k
	\end{align*}
	This completes the proof of the claim and the lemma.
\end{proof}

\begin{proof}[Proof of Lemma \ref{mainlemma}]
	As stated earlier, the proof goes by induction on $d$. For $d=1$, by Lemma \ref{lemmabogolyubov} there is a subspace $U\s \F^{n_1}$ of codimension at most $C(\log 1/\d)^4$ contained in $2\cB'-2\cB'$. Choose $Q=U$. Now if $r.q=0$ for at least $(1-2^{-3^4})|Q|$ choices $q\in Q$ then the same holds for all $q\in Q$, therefore $r\in U^{\perp}$, but $\dim(U^{\perp})\leq C(\log 1/\d)^4$, so the case $d=1$ is proved.
	
	\smallskip  
	
	Now let us assume that $d\geq 2$. Extend the total order $\prec$ defined above such that it now contains $\emptyset$ which has $\emptyset \prec I$ for every non-empty $I\s \lbrack d-1\rbrack$. Say $\emptyset=I_0\prec I_1\prec I_2\prec \dots \prec I_{2^{d-1}-1}$ where $\{I_0,\dots,I_{2^{d-1}-1}\}=P(\lbrack d-1\rbrack)$.
	
	\smallskip
	
	\noindent \emph{Claim.} For every $0\leq i\leq 2^{d-1}-1$ there exists a multiset $Q_{I_i}$ of pure tensors with elements chosen from $2^{3^{d+3}}\cB'-2^{3^{d+3}}\cB'$, and subspaces $W_{I_j}(i)\s \F^{I_j}$, $W_{(I_j)^c}(i)\s \F^{(I_j)^c}$ for every $j\leq i$ (for $j=0$, we only require $W_{\lbrack d\rbrack}(i)$ and not $W_{\emptyset}(i)$) with the following properties. The dimension of each of these spaces is at most $g_1(i)=G(d-1,\d)^{\a(i)}$, where $\a(i)=4 \cdot \Pi_{1\leq j\leq i} \hspace{1mm} 2c_2(|I_j|)$
	. Moreover, if $r\in \cG$ has $r.q=0$ for at least $(1-\frac{1}{4}(2^{-3^{d+2}})^2)|Q_{I_j}|$ choices $q\in Q_{I_j}$ for all $j\leq i$, then $r\in W_{\lbrack d\rbrack}(i)+ \sum_{1\leq j\leq i} (W_{I_j}(i)\otimes \F^{(I_j)^c}+\F^{I_j}\otimes W_{(I_j)^c}(i))+\sum_{j>i}\F^{I_j}\otimes H_{(I_j)^c}(i,r)$ holds for some $H_{(I_j)^c}(i,r)$ possibly depending on $r$ and of dimension at most $g_1(i)$.
	
	\smallskip
	
	\noindent \emph{Proof of Claim.} This is proved by induction on $i$. For $i=0$, by Lemma \ref{focusondeg}, there exist $Q_{\emptyset}\s 2\cB'-2\cB'$ consisting of pure tensors and $V_{\lbrack d\rbrack}\s \F^{\lbrack d\rbrack}$ of dimension at most $4C(\log (2/\d))^4\leq 4C(2\log (1/\d))^4\leq G(d-1,\d)^4$ such that if $r.q=0$ for at least $\frac{7}{8}|Q_{\emptyset}|$ choices $q\in Q_{\emptyset}$, then $r$ can be written as $r=x+y$ where $x\in V_{\lbrack d\rbrack}$ and $y$ is $g_2$-degenerate for $g_2=G(d-1,\frac{\d}{4|\F|^{4C(\log 2/\d)^4}})$. Since
	\begin{align*}
	g_2&\leq G(d-1,|\F|^{-5C(\log 2/\d)^4})=((\log |\F|)^2c_1(d-1)5C(\log (2/\d))^4)^{c_2(d-1)} \\
	&\leq((\log |\F|)^2c_1(d-1)5C(2\log (1/\d))^4)^{c_2(d-1)}\leq G(d-1,\d)^{4},
	\end{align*}
	we can take $W_{\lbrack d\rbrack}(0)=V_{\lbrack d\rbrack}$.
	
	Once we have found suitable sets $W_{I_j}(i-1)$ and $W_{(I_j)^c}(i-1)$ for all $j\leq i-1$, we can apply Lemmas \ref{claim} and \ref{keylemma} with $I=I_i$ and $k=g_1(i-1)$ to find a suitable $Q_{I_i}$, $W_{I_j}(i)$ and $W_{(I_j)^c}(i)$ for all $j\leq i$, and the claim is proved, since $g_1(i)=g_1(i-1)^{2c_2(|I_i|)}$.
	
	\smallskip
	
	Now, after taking several copies of each $Q_{I}$, we may assume that additionally $\max_I |Q_I|\leq 2\min_{I} |Q_I|$. Let $Q=\bigcup_{I\s \lbrack d-1\rbrack}Q_I$ and suppose that $r.q=0$ for at least $(1-2^{-3^{d+3}})|Q|$ choices $q\in Q$. Since $2^{-3^{d+3}}\leq \frac{1}{2\cdot 2^{d-1}}\cdot \frac{1}{4}(2^{-3^{d+2}})^2$, it follows that for every $I\s \lbrack d-1\rbrack$ we have $r.q=0$ for at least $(1-\frac{1}{4}(2^{-3^{d+2}})^2)|Q_I|$ choices $q\in Q_I$. By the Claim with $i=2^{d-1}-1$, we get that $r\in \sum_{I\s \lbrack d\rbrack, I\neq \emptyset} V_I\otimes \F^{I^c}$ for some $V_I\s \F^{I}$ not depending on $r$, and of dimension at most $g_1(2^{d-1}-1)=G(d-1,\d)^{\a(2^{d-1}-1)}$. Note that $$\a(2^{d-1}-1)=4\cdot 2^{2^{d-1}-1}\cdot \Pi_{1\leq i\leq d-1} c_2(i)^{{d-1 \choose i}}.$$
	But $$\Pi_{1\leq i\leq d-1} c_2(i)^{{d-1 \choose i}}=4^{\sum_{1\leq i\leq d-1} {d-1 \choose i}i^{i}}\leq 4^{\sum_{1\leq i\leq d-1} {d-1 \choose i}(d-1)^{i}}\leq 4^{(d-1+1)^{d-1}}=4^{d^{d-1}}.$$
	Thus, $\a(2^{d-1}-1)\leq 4^{d^d}$.
	This completes the proof of the lemma.
\end{proof}



\section*{Acknowledgments} 
I would like to thank Timothy Gowers for helpful discussions. I am also grateful to him and the anonymous referee for their valuable comments on a previous version of this paper.

\bibliographystyle{amsplain}

\begin{thebibliography}{99}


\bibitem{bhowmicklovett}
A. Bhowmick and S. Lovett.
\newblock Bias vs structure of polynomials in
	large fields, and applications in effective algebraic geometry and coding theory.
\newblock 2015, arXiv:1506.02047.

\bibitem{gowersjanzer}
W. T. Gowers and O. Janzer.
\newblock Subsets of Cayley graphs that induce many edges.
\newblock {\em Theory of Computing}, 15(20):1--29, 2019.

\bibitem{gowerswolf}
W. T. Gowers and J. Wolf.
\newblock Linear forms and higher-degree uniformity for
	functions on $\F_p^n$.
\newblock {\em Geometric and Functional Analysis}, 21(1):36--69, 2011.

\bibitem{greentao}
B. Green and T. Tao.
\newblock The distribution of polynomials over finite fields,
	with applications to the Gowers norms.
\newblock {\em Contributions to Discrete Mathematics},
4(2), 2009. 

\bibitem{HS10}
E. Haramaty, and A. Shpilka.
\newblock On the structure of cubic and quartic polynomials.
\newblock {\em Proceedings of the forty-second ACM symposium on Theory of computing}, pp. 331-340. ACM, 2010.

\bibitem{fouriersurvey}
H. Hatami, P. Hatami and S. Lovett.
\newblock Higher-order Fourier Analysis and Applications.

\bibitem{janzerrank}
O. Janzer.
\newblock Low analytic rank implies low partition rank for tensors.
\newblock 2018, arXiv:1809.10931.

\bibitem{kaufmanlovett}
T. Kaufman and S. Lovett.
\newblock Worst case to average case reductions for
	polynomials.
\newblock {\em 49th Annual IEEE symposium on Foundations of Computer Science}, 2008.

\bibitem{KZ18}
D. Kazhdan, T. Ziegler.
\newblock Approximate cohomology.
\newblock {\em Selecta Mathematica}, 24(1):499-509, 2018.

\bibitem{Lam19}
A. Lampert.
\newblock Bias Implies Low Rank for Quartic Polynomials.
\newblock 2019, arXiv:1902.10632.

\bibitem{lovett}
S. Lovett.
\newblock The analytic rank of tensors and its applications.
\newblock {\em Discrete Analysis}, 2019:7, 10pp.

\bibitem{Mil19}
L. Mili\'cevi\'c.
\newblock Polynomial bound for partition rank in terms of analytic rank.
\newblock {\em Geometric and Functional Analysis}, 29:1503--1530, 2019.

\bibitem{naslund}
E. Naslund.
\newblock The partition rank of a tensor and $k$-right corners in $\mathbb{F}_q^n$.
\newblock {\em Journal of Combinatorial Theory, Series A}, 174 (2020) 105190.

\bibitem{sanders}
T. Sanders.
\newblock On the Bogolyubov-Ruzsa lemma.
\newblock {\em Analysis \& PDE}, 5(3):627-55, 2012.

\end{thebibliography}


\begin{dajauthors}
\begin{authorinfo}[janzer]
  Oliver Janzer\\
  University of Cambridge\\
  Cambridge, United Kingdom\\
  oj224\imageat{}cam\imagedot{}ac\imagedot{}uk \\
  \url{https://sites.google.com/view/oliver-janzer/home}
\end{authorinfo}
\end{dajauthors}

\end{document}